\documentclass[12pt,a4paper]{article}
\usepackage[english]{babel}
\usepackage[utf8x]{inputenc}
\usepackage{amsmath}
\usepackage[margin=1in]{geometry}
\usepackage{graphicx}
\usepackage[section]{placeins}
\usepackage{float}
\usepackage{ amssymb }
\usepackage{amsthm}
\usepackage{amsfonts}
\usepackage{epsfig}
\usepackage{ textcomp }
\usepackage{mathtools}
\usepackage{authblk}
\usepackage{etoolbox}

\usepackage[colorlinks=true, citecolor=black, linkcolor=black, urlcolor=blue]{hyperref}
\usepackage[svgnames]{xcolor}

\newcommand{\bigbinom}[2]{\left(\!\!\begin{array}{c}#1\\#2\end{array}\!\!\right)}

\newcommand{\bb}{\begin{equation}}
\newcommand{\ee}{\end{equation}}

 \newtheorem{thm}{Theorem}

 \newtheorem{lem}[thm]{Lemma}

\newcommand{\QED} {\hfill$\square$}

\title{Exactly Colored Complete Subgraphs of Infinite Graphs}
\date{(December 2025)}
\author{\v{Z}arko Ran\dj elovi\'c}

\affil{\small{ Mathematical Institute of the Serbian Academy of Sciences and Arts
\\
Kneza Mihaila 36\\
Belgrade 11000, Serbia}}
\affil{\texttt{zarko.randjelovic@turing.mi.sanu.ac.rs}}

\begin{document}
\maketitle

\begin{abstract}
    Given integers $m\le c$ and an exact $c$-coloring of the edges of a complete countably infinite graph (i.e. a coloring that uses exactly $c$ colors), must there be an infinite subgraph that is exactly $m$-colored? Using the Infinite Ramsey Theorem It is easy to show that the statement is true if $m=1,2$ or $c$. Erickson conjectured that it is false in all other cases. Stacey and Weidl proved that for each $m\ge 3$ there is some large enough $C(m)$ such that the conjecture is true for all pairs $(c,m)$ with $c>C(m)$. The main aim of this paper is to show that for all large enough $m$ the conjecture holds for all $c>m$. This reduces the number of cases needed to fully verify the conjecture to a finite number.
\end{abstract}

\section{Introduction}
Suppose that $c\ge m\ge 1$ are integers and let $P(c,m)$ be the assertion that given any exact $c$-coloring of the edges of a complete countably infinite graph (that is, a coloring with $c$ colors all of which must be used at least once) there exists an exactly $m$-colored countably infinite complete subgraph. Erickson \cite{ericksonconjecture,ERICKSON1994395} conjectured that $P(c,m)$ is true if and only if $m=1,m=2$ or $m=c$. The case $m=c$ is trivial and $m=1$ is the  Infinite Ramsey Theorem. Erickson \cite{ERICKSON1994395} verified that $P(c,m)$ is also true for $m=2$ and also gave examples to show that $P(c,m)$ is false in many other cases.  Even though the pairs covered by these cases have density $1$ in all pairs there are still infinitely many pairs left. Stacey and Weidl \cite{STACEY19991} showed that for all $m$ there is some $C(m)$ such that $P(c,m)$ is false for all $c\ge C(m)$. They also point out that we can cover most cases of other pairs $(c,m)$ by considering modulo $6$. This however misses some classes modulo $6$, still leaving out an infinite set of cases when $c$ is not much larger than $m$.

The goal of this paper is to show that for large enough $m$, $P(c,m)$ is false for any $c>m$. This, combined with the result of Stacey and Weidl, will give only a finite number of cases left to verify the conjecture of Erickson. We first state a well known result in number theory that will be of use.

\begin{thm}
    (Prime Number Theorem) For any real number $x>1$ define $\pi(x)$ to be the number of prime numbers less than or equal to $x$. We then have 
    $$\lim_{x\rightarrow \infty}\frac{\pi(x)}{\frac{x}{\log x}}=1$$
\end{thm}
Our problem can be reduced to constructing finite graphs. It is not hard to show that if we can construct a finite (not necessarily complete) graph whose edges are colored with exactly $c$ colors that has no induced subgraph whose edges use exactly $m$ colors then both $P(c+1,m+1)$ and $P(c+2,m+2)$ are false. We will split the problem into two cases. First the case when $c<\frac{m\sqrt{\log m}}{100}$ and then the case when $c\ge \frac{m\sqrt{\log m}}{100}$. The second case will be done using the same idea as in \cite{STACEY19991} with slight improvements on the bounds so that $c$ can reach an order as low as $m\sqrt{\log m}$. The first case will be done with a construction using the prime number theorem. We will state the two cases as theorems.
\begin{thm}\label{lowc}
    For large enough $m$ if $c< \frac{m(\log m)^{1/4}}{2}$ then there is a finite graph $G$ whose edges are exactly $c$-colored that has no exactly $m$-colored subgraph.
\end{thm}
\begin{thm}\label{largec}
    For large enough $m$ if $c\ge \frac{m(\log m)^{1/4}}{2}$ and there is a positive integer $k$ such that $\binom{k}{2}-m$ is a non-negative odd number less than $k-1$ then there is a finite graph $G$ whose edges are exactly $c$-colored that has no exactly $m$-colored subgraph.
\end{thm}
These two results will imply our main theorem of this paper.
\begin{thm}\label{mainthm}
    For large enough $m$ if $c>m$ then $P(c,m)$ is false.
\end{thm}
In section $2$ we will prove Theorem \ref{lowc}, in section 3 we will prove Theorem \ref{largec} and in the final section we will describe how these theorems give Theorem \ref{mainthm}. Throughout this paper we will use standard notation for graphs and their parameters. See Bollob\'{a}s \cite{10.5555/7228} for general background.

\section{Proof of Theorem 2}

The general strategy of the proof will be as follows. We will construct a complete bipartite graph $K_{a,b}$ with bipartition $(V,W)$ with $|V|=a,|W|=b$ and use a separate color for each edge. In addition we will add all edges between vertices of $W$ and color them in a relatively small number $d$ of colors so that in total we have $c$ colors and any large subset of $W$ contains all $d$ colors used within $W$. We will also have $a$ being quite small. The goal is to ensure that in order to have exactly $m$ colors we need to pick a large subset of $W$, hence requiring exactly $m-d$ colors used in $K_{a,b}$. In order to get a contradiction we will pick $d$ to be such that $m-d$ does not have any divisors less than $a$ that are larger than $\frac{50a}{\sqrt{\log m}}$. This will ensure that we cannot have exactly $m$ colors in total. 

We will first prove a few lemmas before moving on to the proof of Theorem \ref{lowc}. Throughout this section all products over $p$ are assumed to be taken over prime numbers $p$.
\begin{lem}\label{productabovetenth}
    For $s$ large enough we have that $$\prod_{\frac{s}{10}< p\le s}p>\left(\frac{s}{10}\right)^{\frac{0.5s}{\log s}}$$
\end{lem}
\begin{proof}
  By the prime number theorem for $s$ large enough $$\pi\left(\frac{s}{10}\right)<1.1\frac{s}{10\log \left(\frac{s}{10}\right)}<1.1\frac{s}{\frac{10}{2}\log s}=0.22\frac{s}{\log s}.$$ Also by the prime number theorem we have that for large enough $s$, $\pi (s)>0.9\frac{s}{\log s}$ so for $s$ large enough $\pi(s)-\pi\left(\frac{s}{10}\right)>0.5\frac{s}{\log s}$ so $$\prod_{\frac{s}{10}< p\le s}p>\left(\frac{s}{10}\right)^{\frac{0.5s}{\log s}}$$ as required.
\end{proof}

\begin{lem}\label{ratiobound}
    For $s\in \mathbb{R}_{\ge 0}$ large enough we have that $$\frac{\prod\limits_{\frac{s}{10}<p\le s}p}{\prod\limits_{p\le \frac{s}{10}}p}>s^{40}$$
\end{lem}
\begin{proof}
    By the prime number theorem we have that  
    \[
    \prod_{p\le \frac{s}{10}}p<\left(\frac{s}{10}\right)^{\frac{1.1s/10}{\log (s/10})}<\left(\frac{s}{10}\right)^{\frac{1.1s}{10\cdot \frac{1}{2}\log s}}=\left(\frac{s}{10}\right)^{\frac{2.2s}{10\log s}}<\left(\frac{s}{10}\right)^{\frac{1}{4}\frac{s}{\log s}}
    \]
    From Lemma \ref{productabovetenth}
    $$\frac{\prod_{\frac{s}{10}<p\le s}p}{\prod_{p\le \frac{s}{10}}p}>\frac{\left(\frac{s}{10}\right)^{\frac{0.5s}{\log (s)}}}{\left(\frac{s}{10}\right)^{\frac{0.25s}{\log (s)}}}=\left(\frac{s}{10}\right)^{\frac{0.25s}{\log (s)}}>s^{\frac{0.125s}{\log s}}>s^{40}$$ for large enough $s$ as required.
\end{proof}

\begin{lem}\label{colorfulcliques}
    Let $k,l,s$ be positive integers with $l\ge 2,k\le l\le s$ and $e^2s<le^{\frac{l}{4k}}$ (where $e$ denotes Euler's number), and let $K_s$ be the complete graph on $s$ vertices. Then we can color the edges of $K_s$ using exactly $k$ colors such that every induced subgraph of $K_s$ containing $l$ vertices has at least one edge from every color.
\end{lem}
\begin{proof}
    Note that the lemma is trivial when $k=1$ so we may assume that $k>1$. Suppose that the set of colors is $[k]$. Color the graph $K_s$ randomly by assigning a color from $[k]$ to each edge independently with equal probability for each color. Given a color $i$ and a subset $A\subset K_s$ of size $l$ the probability that no edge in $K_s[A]$  has color $i$ is $$p_i=\left(\frac{k-1}{k}\right)^{\binom{l}{2}}.$$ Therefore by the union bound the probability that all edges in $K_s[A]$ are colored in at most $k-1$ colors is at most $\sum_{i=1}^k p_i= k\left(\frac{k-1}{k}\right)^{\binom{l}{2}}$. There are $\binom{s}{l}$ subsets of $K_s$ of size $l$ so again by the union bound the probability that there is at least one subset $A\subset K_s$ of size $l$ such that the edges of $K_s[A]$ are colored in less than $k$ colors is at most 
    \begin{align}\label{l71}
        p=\binom{s}{l}k\left(\frac{k-1}{k}\right)^{\binom{l}{2}}.
    \end{align}
         Note that clearly $\binom{n}{t}\le n^t/t!$ for any integers $0\le t\le n$ where $0!=1$. We will first prove that $n!>n^n/e^n$. Indeed, this is true for $n=1$ and note that for any positive integer $n$ we have that $$\left(1+\frac{1}{n}\right)^{n}= \sum_{i=0}^n\frac{1}{n^i}\binom{n}{i}\le \sum_{i=0}^n\frac{1}{i!}<\sum_{i=0}^{\infty}\frac{1}{i!}=e.$$ Thus we have that $(n+1)^n/e<n^n$. So, if $n!>n^n/e^n$ we have that $$(n+1)!>\frac{(n+1)n^n}{e^n}>\frac{(n+1)^{n+1}}{e^{n+1}}.$$ This proves by induction that $n!>n^n/e^n$ for all positive integers $n$. Now by (\ref{l71}) we have that 
         \begin{align}\label{l72}
             p\le \frac{s^l}{l!}k\left(\frac{k-1}{k}\right)^{\frac{l(l-1)}{2}}<\left(\frac{es}{l}k^{1/l}\left(\frac{k-1}{k}\right)^{\frac{l-1}{2}}\right)^l
         \end{align}
         Note that $$\left(1+\frac{1}{k-1}\right)^{\frac{l-1}{2}}=\left(1+\frac{1}{k-1}\right)^{k-1\frac{l-1}{2(k-1)}}\ge 2^{\frac{l-1}{2(k-1)}}\ge e^{\frac{l}{4k}}$$ since $l\ge k$. We also have that $k^{1/l}\le e^{\frac{\log k}{k}}<e$ so by (\ref{l72}) and the condition $e^2s<le^{\frac{1}{4k}}$ we have that $$p< \left(\frac{e^2s}{le^{\frac{l}{4k}}}\right)^{l}<1.$$ Thus there exists a coloring satisfying the required condition. This proves the lemma.
\end{proof}

We now move on to the proof of Theorem \ref{lowc}.
\begin{proof}[Proof of Theorem \ref{lowc}.]
    For each positive integer $l$ define $s(l)$ to be $$s(l)=\prod_{\frac{\sqrt{l}}{10}<p\le \sqrt{l}}p^2\prod_{\sqrt{l}< p\le l}p.$$
    Take $t$ to be the smallest integer such that 
    \[
    s(t)>\frac{m(\log m)^{1/4}}{2}.
    \]Such a $t$ certainly exists (for example if $t$ is a very large prime it satisfies the above inequality). Note that $$s(t)<\left(\prod_{p\le t}p\right)^2\le t^{2\pi(t)}<t^{2t}$$
    We now show that $t> \frac{\log m}{2\log \log m}$ for large $m$. If this were not true we would have $t< \log m$ and $2t\le \frac{\log m}{\log \log m}$. Therefore $$t^{2t}<e^{\log \log m\frac{\log m}{\log \log m}}=m<s(t)$$ which is a contradiction. Thus we have that 
    \begin{align}\label{tlowerbound}
    t> \frac{\log m}{2\log \log m}>100\sqrt{\log m}
    \end{align}
    for $m$ large enough and hence $t\rightarrow \infty$ as $m\rightarrow \infty$.
    We also need an upper bound for $t$. Let us show that $t\le 10 \log m$. Take $t'$ to be any integer such that $9\log m<t'<10\log m$ (where $m>3$). For $m$ large enough $\frac{t'}{10}>\sqrt{t'}$ and also $\log t'<2\log \log m$. By Lemma \ref{productabovetenth} $$s(t')>\prod_{\frac{t'}{10}<p\le t'}p>\left(\frac{t'}{10}\right)^{\frac{0.5t'}{\log t'}}>t'^{\frac{0.25t'}{\log t'}}>(\log m)^{\frac{0.25\cdot 9\log m}{2\log \log m}}=m^{\frac{9}{8}}>\frac{m(\log m)^{1/4}}{2}.$$ Thus by definition of $t$ we must have that $t<10\log m$. Note that for any positive integer $l>3$, $\frac{s(l)}{s(l-1)}\le l^2$ since you can add at most one new prime to $s(l)$ and at most one prime increases its power from one to two in $s(l)$ compared to $s(l-1)$. Since $s(t-1)\le \frac{m(\log m)^{1/4}}{2}$ we have that \begin{align}
        s(t)<m(\log m)^3.
    \end{align}
   for large $m$. Note that $t^2>\log m$ by (\ref{tlowerbound}). By Lemma \ref{ratiobound} we obtain that for large $m$
   $$\frac{s(t)}{\prod\limits_{p\le t}p}=\frac{\prod\limits_{\frac{\sqrt{t}}{10}<p\le \sqrt{t}}p}{\prod\limits_{p\le \frac{\sqrt{t}}{10}}p}>t^{20}>\log ^{10}m.$$ Hence 
   $\prod\limits_{p\le t}p<\frac{s(t)}{\log^{10}m}<\frac{m}{\log^{7}m}$. 

   Since $s(t)>c$ there must be a prime power in the product $s(t)$ that does not divide $c-m$. Let that be $n=p_1^{\epsilon}$ where $\epsilon=1$ or $2$. We also have that $n\le t$ so there exists a  positive $r\le t$ such that $n|c-r$. Now, we have that $n\nmid m-r$. Since $\gcd (n,\prod\limits_{\substack{p<t\\ p\neq p_1}}p)=1$ there is a non-negative $x< \prod\limits_{\substack{p\le t\\ p\neq p_1}}p<\frac{m}{\log ^{7}m}$ such that $\prod\limits_{\substack{p\le t\\ p\neq p_1}}|m-r-nx-1$. Namely we have that $m-r-nx$ is not divisible by any prime lower than $t$ other than possibly $p_1$. Since $r,n\le t<10\log m$ we obtain $r+nx<10\log m+\frac{10m}{\log ^{6}m}<\frac{m}{\log ^{5}m}$ for large $m$. Let $a=n,b=\frac{c-r-nx}{n}$ and let the graph $G$ consist of the vertex set $V\cup W$ where $|V|=a,|W|=b$. Define the set $E(G)$ of edges of $G$ to consist of all pairs of vertices containing at least one vertex from $W$. We will color all edges of $G$ with a color in $[c]$. Color all edges containing one vertex from $V$ with distinct colors $1,2,\ldots ab$. We will use the remaining $r+nx$ colors to color edges with both vertices in $W$. Let $s=b,l=\left[\frac{m}{20\log m}\right]$ and $k=r+nx$. Note that $r+nx<\frac{m}{\log^5m}<\left[\frac{m}{20\log m}\right]$ and $c-r-nx>m/2$ so $b>\left[\frac{m}{20\log m}\right]$. Thus we have that $k<l<s$. We also trivially have $k\ge 1,l\ge 2$. Now we check the final condition for Lemma \ref{colorfulcliques}. Since $\frac{l}{4k}\ge \frac{m}{100\log m}\frac{\log^5 m}{m}>\log^3 m$ for large $m$ we have that $le^{\frac{l}{4k}}>m^{\log^2 m}$. But we also have that $e^2s<ce^2<m^2$ and hence $e^2s<le^{\frac{l}{4k}}$. Thus we can apply Lemma \ref{colorfulcliques} on the complete graph $G[W]$ with $s,l,k$ as above to obtain a coloring of the edges of $G[W]$ such that for any $A\subset W$ of size at least $\frac{m}{20\log m}$ the edges in $G[A]$ use all $r+nx$ colors. We use this coloring of $G[W]$ where the $k$ colors are precisely the remaining colors $ab+1,ab+2,\ldots ,c$. \\
   \\
   It remains to show that no induced subgraph of this graph is exactly $m$-colored. Suppose the contrary and let $A\subset G$ be such that the edges of $G[A]$ have exactly $m$ distinct colors Note that less than $m/2$ colours are used in $G[W]$ so $G[A]$ must use at least $m/2$ of the first $c$ colors. Now let $|A\cap V|=v,|A\cap W|=w$. We have that $vw\ge m/2$. Thus $w\ge \frac{m}{2v}\ge \frac{m}{2a}\ge \frac{m}{20\log m}$. Hence by our coloring of $W$ all of the last $r+nx$ colors are used within $A$ and thus $vw=m-r-nx$. Note that since $m-r-nx$ has no prime divisors less than $t$ other than possibly $p_1$ and since $n\nmid m-r$ and $v\le t$ we must have that $v<n$ and $n/v$ is either $p_1$ or $p_1^2$. From (\ref{tlowerbound}) we obtain $ p_1\ge \sqrt{t}/10> (\log m)^{1/4}$ and hence the number of colors used in $A$ among the first $ab$ colors is at most $\frac{ab}{p_1}< \frac{c}{(\log m)^{1/4}}< \frac{m}{2}$ giving a contradiction. This proves the theorem.

\end{proof}
\section{Proof of Theorem \ref{largec}}

In this section we prove Theorem \ref{largec}. The proof will be very similar to that of the proof of Theorem 3 in \cite{STACEY19991}. We will just slightly modify the proof of Lemma 4 in \cite{STACEY19991} weakening one condition but strengthening another. This will allow us to close the gap between low and large values of $c$ and cover all pairs $(c,m)$ for large enough $m$. Together with Theorem 3 in \cite{STACEY19991} this will leave only a finite (albeit large) number of remaining unsolved cases of Erickson's conjecture. We begin with a slight modification of Lemma 4 in \cite{STACEY19991}.

\begin{lem}\label{prob4sets}
    There is a constant $\alpha>1$ such that if $n,l$ are integers with $l\ge 100$ and $n\ge \alpha l$ there are $4$-sets $X_1,X_2,\ldots X_r\subset [n]$, where $r=\lceil3 n/2\rceil$, satisfying the following properties. For any $i\neq j$ we have $|X_i\cap X_j|\le 1$ and for any $V\in [n]^{(l)}$ we have that $$|\{i\in [r]\ |X_i\subset V\}|\le 2l/5$$
\end{lem}

\begin{proof}
    Let $n,l$ be integers where $l\ge 100$ and $n\ge \alpha l$, where $\alpha>1$ will be chosen later. Since the proof is similar to Lemma 4 in \cite{STACEY19991} we will skip over some of the details that are obtained in an identical fashion. We will refer to \cite{STACEY19991} in  such cases. Choose a random collection of $4$-sets that are subsets of $[n]$ where each set is chosen independently with probability $p$ where $$p=\frac{60}{(n-1)(n-2)(n-3)}.$$ Let $\mathcal{A}$ be this random collection. As in \cite{STACEY19991} since the mean and variance of $|\mathcal{A}|$ are $5n/2$ and $5n(1-p)/2$ we obtain by Chebyshev's inequality that 
    \begin{align}\label{prob1}
        \mathbb{P}\left(\Big| |\mathcal{A}|-5n/2\Big|\ge n/2\right)\le \frac{5n(1-p)2}{n^2}<1/6.
    \end{align}
    Now, define 
    $$\mathcal{I}=\{(A,B)\ |\ A,B\in \mathcal{A},|A\cap B|\ge 2\}.$$
    Just like in \cite{STACEY19991} we obtain that
\[
\mathbb{E}(|\mathcal{I}|)=\frac{n(n-1)(n-2)^2(n-3)^2}{16}p^2=\frac{225n}{n-1}<250
\] 
so by Markov's inequality we have that 
    \begin{align}\label{prob2}
        \mathbb{P}(|\mathcal{I}|\ge n/2)\le \frac{500}{n}<1/6
    \end{align}
    when $\alpha>30$.
    Suppose that $X\in [n]^{(l)}$. Again, similar to \cite{STACEY19991} we have 
    $$\mathbb{P}(|X^{(4)}\cap \mathcal{A}|\ge \lceil 2l/5\rceil )\le \binom{\binom{l}{4}}{\lceil2l/5\rceil}\left(\frac{60}{(n-1)(n-2)(n-3)}\right)^{\lceil2l/5\rceil}\le \frac{\beta}{(n-3)^{6l/5}}$$ where $\beta =\bigbinom{\binom{l}{4}}{\lceil2l/5\rceil}60^{\lceil2l/5\rceil}$.
    Using the bound $\binom{a}{b}\le \frac{a^b}{b!}\le \frac{(ea)^b}{b^b}$ we get that 
\begin{align}\label{betaineq}
\beta\le \left(\frac{el^4}{5l}\right)^{2l/5+1}60^{2l/5+1}\le (5l)^{6l/5}60l^3<(n-3)^{6l5}
\end{align}
where we use that $l^4>5l<24\lceil2l/5\rceil$ and $\alpha >15$.
Now by Harris's inequality ({}\cite{Harris_1960}; see also \cite{Liggett}) we obtain that 
    \begin{align}\label{betabound}
        \mathbb{P}(\forall X\in [n]^{(l)},\ \  |X^{(4)}\cap\mathcal{A}|< \lceil2l/5\rceil)\ge \left(1-\frac{\beta}{(n-3)^{6l/5}}\right)^{\binom{n}{l}}\ge 1- \frac{2^{6l/5}\beta}{n^{6l/5}}\binom{n}{l}.
    \end{align}
From the first inequality in (\ref{betaineq})
\[
\frac{2^{6l/5}\beta}{n^{6l/5}}\binom{n}{l}\le \frac{(8el^3\cdot 12)^{2l/5}(en)^l96el^3}{n^{6l/5}l^l}=(96^2e^7)^{l/5}\left(\frac{l}{n}\right)^{l/5}96el^3<1/2
\] 
if we choose $\alpha=10^{10}$. Now using (\ref{prob1}),(\ref{prob2}) and (\ref{betabound}) we obtain that there is a family $\mathcal{A}$ such that $|\mathcal{I}|<n/2,|\mathcal{A}|\ge 2n$ and for any $X\in [n]^{(l)}$ we have that $|X^{(4)}\cap \mathcal{A}|<2l/5$. Finally, if we remove a set from each of the pairs in $\mathcal{I}$ (and potentially even more sets to have exactly $r=\lceil3n/2\rceil$ remaining) we obtain the desired family.
    \end{proof}
We are now ready to prove Theorem \ref{largec}. Again the proof will be very similar to \cite{STACEY19991}.\\

\textit{Proof of Theorem \ref{largec}.} Suppose that $m>e^{10^{83}}$. Write $c,m$ in the form $c=\binom{r'}{2}-p',m=\binom{k}{2}-q$ where $0\le p'<r'-1,0\le q<k-1$. Now, if $p'$ is even set $r_1=r',p=p'$ and otherwise set $r_1=r'+2,p=p'+2r'+1$. By the assumption of the theorem $q$ is odd and it is also clear that $p$ is even. Note that $k\le 2\sqrt{m}-1$ and $r_1\ge \sqrt{m}(\log m)^{1/8}$. Applying Lemma \ref{prob4sets} for $l=k+1$ with $n=r_1$ we obtain a collection $X_1,X_2,\ldots X_r\subset [n]$ of $4$-sets, where $r=\lceil3n/2\rceil$, such that $|X_i\cap X_j|\le 1$ for any $i\neq j$ and no set of size $k+1$ contains more than $2(k+1)/5$ of those $4$-sets. \\

Let $G=K_n$ be the complete graph whose vertex set is $[n]$. We will color the edges of $G$ in $c$ colors as follows. For $1\le i\le p/2$, if $X_i=\{a_i,b_i,c_i,d_i\}$ use color $2i-1$ for edges $a_ib_i,c_id_i$ and use color $2i$ for edges $a_ic_i,b_id_i$. This is possible since $p\le 3r'\le 3n$. Color all the other edges with the remaining unused colors so that each of those colors is used exactly once. We now claim that no induced subgraph is exactly $m$-colored. Suppose the contrary and let $X\subset [n]$ such that $G[X]$ is exactly $m$-colored.\\

Clearly $|X|\ge k$. Notice that $G[X]$ has exactly $\binom{|X|}{2}-s$ colors, where $s$ is the number of colors $1,2,\ldots ,p$ that have two edges within $X$. We also have that for each $1\le i\le p/2$ color $2i$ appears twice if and only if color $2i-1$ appears twice which is if and only if $X_i\subset X$. This implies that $s$ must be even. Now we get $|X|>k$ as otherwise $s=q$ which is odd. Take any $X'\subset X$ of size $k+1$. By construction from Lemma \ref{prob4sets} we have that the number of $4$-sets contained within $X'$ is at most $2(k+1)/5$. Hence $G[X']$ has at least $\binom{k+1}{2}-4(k+1)/5>\binom{k}{2}\ge m$ colors since $k>4$. This is a contradiction, which completes the proof. \QED

\section{Proof of Theorem 1 and Conclusion}

We now use theorems \ref{lowc} and \ref{largec} to derive Theorem \ref{mainthm}. Again, the proof is analogous to the proof in \cite{STACEY19991} but we will include it for completeness. Suppose that $m$ is large and $c>m$. Now let $k,q$ be non-negative integers such that $m=\binom{k}{2}+1-q$ and $q\le k-2$. Let $X$ be the complete graph on the vertex set $\mathbb{N}$.\\

If $q$ is odd then let $c'=c-1,m'=m-1$. Now apply theorem \ref{lowc} or \ref{largec} to $(c',m')$ depending on whether $c'$ is below or above $\frac{m'(\log m')^{1/4}}{2}$. When doing this the graph we color will be some finite subgraph of $X$. Now color the rest of the edges of $X$ with a new color. It is easy to see that this graph satisfies the required property.\\

If $q$ is even and $q<k-2$ then let $c'=c-2,m'=m-2$. Again, similar to the previous case, we may apply theorem \ref{lowc} or \ref{largec} to $(c',m')$ to obtain a coloring of some finite subgraph $G\subset X$ with colors in $[c']$. Now we color all edges with exactly one vertex in $G$ with color $c-1$, and color every remaining uncolored edge with color $c$. Again, it is easy to see that this coloring gives the desired property.\\

The only remaining case is when $q$ is even and $q=k-2$. Stacey and Weidl in \cite{STACEY19991} noted that they have an example for this case without stating it explicitly. We will include one possible example here. Let $c'=c-1,m'=m-1$ and let $n$ be the smallest integer such that $\binom{n}{2}\ge c'$. We will color the complete $K_n$ on some vertex set $\{x_1,x_2,\ldots ,x_n\}\subset X$ with exactly $c'$ colors as follows. For each $1\le i\le n-2$ let $X_i=\{x_i,x_{i+1},x_{i+2},x_{i+3}\}$ and let $Y_i=\{x_ix_{i+1},x_{i+1}x_{i+2}\}$, where $x_{n+1}=x_1$. Now for each $1\le i\le \binom{n}{2}-c'$ color the two edges in $Y_i$ with color $i$. Color the rest of the edges in $K_n$ with the remaining unused colors such that each of these edges gets a different color. This gives a coloring of $K_n$ with exactly $c'$ colors. Now color all the remaining edges of $X$ with color $c$.\\

We claim that no infinite subgraph of $X$ is exactly $m$-colored. Suppose the contrary. It is clear that there must be some $S\subset K_n$ which is exactly $m'$-colored. Note that $|S|\ge k$ so consider any $S'\subset S$ of size $k$. The number of colors used in $S'$ is equal to $\binom{k}{2}-r$ where $r$ is the number of sets $X_i$ with $1\le i\le \binom{n}{2}-c'$ contained within $S$. Since $q=k-2$ we must have that $r\ge k-2$. Let $i_1<i_2<\ldots <i_r$ be such that $X_{i_j}\subset S'$ for each $1\le j\le r$. Now we have that $x_{i_1},x_{i_2},\ldots x_{i_r},x_{i_r+1},x_{i_r+2}$ are distinct elements of $S'$, but $r\ge k-2$ so these must be all elements of $S'$ and $r=k-2$. Since $x_{i_j+1}\in S'$ for any $j$ we must have that $i_1,i_2,\ldots i_r$ are consecutive integers. Also $x_{i_r+3}\in S'$ and hence $i_r+3=n+1,i_1=1$. This implies that  $n-2=r\le \binom{n}{2}-c'\le n-2$ so we have equality everywhere. Therefore, $n=k$ and $c=m=\binom{k}{2}-(k-2)$ which is a contradiction. This completes the proof of Theorem \ref{mainthm}.\QED \\

Combining the result of Theorem \ref{mainthm} together with the results in \cite{STACEY19991} we can see that Erickson's conjecture has now been proven for all but finitely many cases. The bounds in this paper have not been optimized. Perhaps there is some hope that with careful optimization of the bounds and maybe a clever computer method there is a way to verify the conjecture for all pairs $c\ge m\ge 3$.

\bibliographystyle{plain}
\bibliography{main}

\end{document}